\documentclass[12pt,a4paper]{amsart}

\usepackage{amsmath,amsrefs,amsthm}
\usepackage{tikz-cd}
\usepackage{enumitem}
\usepackage[capitalize]{cleveref}
\usepackage{newtxtext}

\newcommand{\T}{\mathcal{T}}
\newcommand{\C}{\mathcal{C}}
\newcommand{\Hom}[3]{\operatorname{Hom}_{#1}(#2,#3)}

\newcommand{\Ker}[1]{\operatorname{Ker}(#1)}
\newcommand{\Coker}[1]{\operatorname{Coker}(#1)}
\newcommand{\Z}{\mathbb{Z}}
\newcommand{\Q}{\mathbb{Q}}
\newcommand{\indT}{\operatorname{ind}(\mathcal{T})}
\newcommand{\indC}{\operatorname{ind}(\mathcal{C})}
\newcommand{\EXT}{\operatorname{Ex}(\mathcal{T})}
\newcommand{\ART}{\operatorname{AR}(\mathcal{T})}
\newcommand{\EXC}{\operatorname{Ex}(\mathcal{C})}
\newcommand{\ARC}{\operatorname{AR}(\mathcal{C})}
\newcommand{\EXCu}{\operatorname{Ex}(\mathcal{\underline{C}})}
\newcommand{\ARCu}{\operatorname{AR}(\mathcal{\underline{C}})}
\newcommand{\F}{\mathcal{F}}

\newtheorem{theorem}{Theorem}[section]
\newtheorem{lemma}[theorem]{Lemma}
\newtheorem{corollary}[theorem]{Corollary}

\theoremstyle{definition}
\newtheorem*{acknowledgements}{Acknowledgements}

\title[\resizebox{4.7in}{!}{Auslander--Reiten triangles and Grothendieck groups of triangulated categories}]{Auslander--Reiten triangles and Grothendieck groups of triangulated categories}

\author{Johanne Haugland}

\begin{document}

\keywords{Auslander--Reiten triangle, Grothendieck group, triangulated category, Frobenius category}
\subjclass[2010]{18E30, 18F30 (primary); 18E10, 16G70 (secondary)} 

\address{Department of mathematical sciences, NTNU, NO-7491 Trondheim, Norway}
\email{johanne.haugland@ntnu.no}

\begin{abstract} We prove that if the Auslander--Reiten triangles generate the relations for the Grothendieck group of a Hom-finite Krull--Schmidt triangulated category with a (co)generator, then the category has only finitely many isomorphism classes of indecomposable objects up to translation.  This gives a triangulated converse to a theorem of Butler and Auslander--Reiten on the relations for Grothendieck groups. Our approach has applications in the context of Frobenius categories.
\end{abstract}

\maketitle

\section{Introduction}

The notion of almost split sequences was introduced by Auslander and Reiten in \cite{AR2}, and has played a fundamental role in the representation theory of finite dimensional algebras ever since \cite{ARS}. The theory of almost split sequences, later called Auslander--Reiten sequences or just AR-sequences, has also greatly influenced other areas, such as algebraic geometry and algebraic topology \cites{auslander,jorgensen}.

Happel defined Auslander--Reiten triangles in triangulated categories \cite{happel2}. These play a similar role in the triangulated setting as AR-sequences do for abelian or exact categories. While it is known that AR-sequences always exist in the category of finitely generated modules over a finite dimensional algebra, the situation in the triangulated case turns out to be more complicated, and the associated bounded derived category will not necessarily have AR-triangles. In fact, Happel proved that this category has AR-triangles if and only if the algebra is of finite global dimension \cites{happel2,happel}. Reiten and van den Bergh showed that a Hom-finite Krull--Schmidt triangulated category has AR-triangles if and only if it admits a Serre functor \cite{RVdB}. More recently, Diveris, Purin and Webb proved that if a category as above is connected and has a stable component of the Auslander--Reiten quiver of Dynkin tree class, then this implies existence of AR-triangles \cite{DPW}. 

In the abelian setting, there is a well-studied relationship between AR-sequences, representation-finiteness and relations for the Grothendieck group. From Butler \cite{butler}, Auslander--Reiten \cite[Proposition 2.2]{AR} and Yoshino \cite[Theorem 13.7]{yoshino}, we know that if a complete Cohen--Macaulay local ring is of finite representation type, then the Auslander--Reiten sequences generate the relations for the Grothendieck group of the category of Cohen--Macaulay modules. Here we say that our ring is of finite representation type if the category of Cohen--Macaulay modules has only finitely many isomorphism classes of indecomposable objects. A converse to this theorem is given by Auslander for artin algebras \cite{auslander1984} and by Hiramatsu in the case of a Gorenstein ring with an isolated singularity \cite[Theorem 1.2]{hiramatsu}, where the latter is extended by Kobayashi \cite[Theorem 1.2]{kobayashi}. Results of the type described above were recently generalized to the setup of exact categories by Enomoto \cite{enomoto} and to certain extriangulated categories by Padrol, Palu, Pilaud and Plamondon \cite{PPPP}.

A natural question to ask is whether there is a similar connection between AR-triangles, representation-finiteness and the relations for the Grothendieck group in the triangulated case. Xiao and Zhu give a partial answer to this question. Namely, they show that if our triangulated category is locally finite, then the AR-triangles generate the relations for the Grothendieck group \cite[Theorem 2.1]{XZ}. Beligiannis generalizes and gives a converse to this result for compactly generated triangulated categories \cite[Theorem 12.1]{beligiannis}.

In this paper we consider the reverse direction of Xiao and Zhu from a different point of view. We prove that if the Auslander--Reiten triangles generate the relations for the Grothendieck group of a Hom-finite Krull--Schmidt triangulated category with a (co)generator, then the category has only finitely many isomorphism classes of indecomposable objects up to translation. We conclude by an application in the context of Frobenius categories. As an example, we see that our approach recovers results of Hiramatsu and Kobayashi for Gorenstein rings.

\section{Auslander--Reiten triangles and Grothendieck groups}

Let $R~$ be a commutative ring. An $R$-linear category $\T$ is called \textit{Hom-finite} provided that $\Hom{\T}{X}{Y}$ is of finite $R$-length for every pair of objects $X, Y$ in $\T$. An additive category is called a \textit{Krull--Schmidt category} if every object can be written as a finite direct sum of indecomposable objects having local endomorphism rings. In a Krull--Schmidt category, it is well known that every object decomposes essentially uniquely in this way.

Throughout the rest of this paper, we let $\T$ be an essentially small $R$-linear triangulated category. We also assume that $\T$ is a Krull--Schmidt category which is Hom-finite over $R$. We let $\indT$ consist of the indecomposable objects of $\T$, while the translation functor of $\T$ is denoted by $\Sigma$. For simplicity, we use the notation $(A,B)= \Hom{\T}{A}{B}$ and \mbox{$[A,B]=\operatorname{length}_R(\Hom{\T}{A}{B})$}.

We say that $\T$ has finitely many isomorphism classes of indecomposable objects up to translation if there is a finite subset of $\indT$ such that for any $U\in\indT$, there is an integer $n$ such that $\Sigma^n U$ is isomorphic to an object in our finite subset.

Recall from \cite{happel3} that a distinguished triangle 
$~A \rightarrow B \xrightarrow{g} C \xrightarrow{h} \Sigma A~$ in $\T$ is an \textit{Auslander--Reiten triangle} if the following conditions are satisfied:
\begin{enumerate}
    \item $A,C\in\indT$;
    \item $h \neq 0$;
    \item given any morphism $t\colon W \rightarrow C$ which is not a split-epimorphism, there is a morphism $t^\prime\colon W \rightarrow B$ such that $g \circ t^\prime = t$.
\end{enumerate}

Let $\F(\T)$ denote the free abelian group generated by all isomorphism classes $[A]$ of objects $A$ in $\T$, while $K_0(\T,0)$ is the quotient of $\F(\T)$ by the subgroup generated by the set $\{[A\oplus B] - [A] - [B] \mid A,B\in\T\}$. By abuse of notation, objects in $K_0(\T,0)$ are also denoted by $[A]$. As $\T$ is a Krull--Schmidt category, the quotient $K_0(\T,0)$ is isomorphic to the free abelian group generated by isomorphism classes of objects in $\indT$. 

Let $\EXT$ be the subgroup of $K_0(\T,0)$ generated by the subset
\[
  \left\lbrace [X]-[Y]+[Z] \;\middle|\;
  \begin{tabular}{@{}l@{}}
   there exists a distinguished triangle\\
$X \rightarrow Y \rightarrow Z \rightarrow \Sigma X~$ in $\T$
   \end{tabular}
  \right\rbrace.
\]
Similarly, we let $\ART$ denote the subgroup of $K_0(\T,0)$ generated by
\[
  \left\lbrace [X]-[Y]+[Z] \;\middle|\;
  \begin{tabular}{@{}l@{}}
   there exists an AR-triangle\\
$X \rightarrow Y \rightarrow Z \rightarrow \Sigma X~$ in $\T$
   \end{tabular}
  \right\rbrace.
\]
Recall from for instance \cite{happel3} that the Grothendieck group of $\T$ is defined as \mbox{$K_0(\T)=K_0(\T,0)/\EXT$}.

In the proof of our main results, \cref{main result} and \cref{main result2}, we use the well-known fact that an equality in $K_0(\T,0)$ can yield an equality in $\Z$. We need this in the case of $[U,-]$ and $[-,U]$ for an object $U$ in $\T$, but note that the following lemma could be phrased more generally in terms of additive functors. 

\begin{lemma} \label{likning bevares}
	Suppose that $~a_1[X_1]+\cdots+a_r[X_r]=0$ in $K_0(\T,0)~$ for integers $a_i$ and objects $X_i$ in $\T$. Then $~a_1[U,X_1]+\cdots+a_r[U,X_r]=0~$ and $~a_1[X_1,U]+\cdots+a_r[X_r,U]=0~$ in $~\Z~$ for any object $U$ in $\T$.
\end{lemma}

\begin{proof}
	Let $a_1[X_1]+\cdots+a_r[X_r]=0$ in $K_0(\T,0)$. If $a_i \geq 0$ for every $i = 1,2,\dots,r$, we use the defining relations for $K_0(\T,0)$ to obtain 
	\[
	a_1[X_1]+\cdots+a_r[X_r]=[a_1X_1 \oplus \cdots \oplus a_rX_r]=0,
	\]
	where $a_iX_i$ denotes the coproduct of the object $X_i$ with itself $a_i$ times. Consequently, the object $~a_1X_1 \oplus \cdots \oplus a_rX_r~$ is zero in $\T$. Applying $[U,-]$ or $[-,U]$ and using additivity hence yields our desired equations. 
	
	If some of the coefficients $a_i$ are negative, we start by moving all negative terms to the right-hand side of our equality and proceed similarly.
\end{proof}

The lemmas below, which yield a triangulated analogue of \cite[Proposition 2.8]{kobayashi}, provide an important step in the proofs of \cref{main result} and \cref{main result2}. Note that parts of our proof of \cref{Lemma AR part 1} is much the same as the proof of \cite[Lemma 2.2]{DPW}. Observe also that \cref{Lemma AR part 2} follows from \cite[Proposition 3.1]{Webb} in the case where $R~$ is an algebraically closed field, and that the argument generalizes to our context. We include complete proofs for the convenience of the reader.

\begin{lemma} \label{Lemma AR part 1}
	Let $A \xrightarrow{f} B \xrightarrow{g} C \rightarrow \Sigma A$ be an AR-triangle in $\T$. The following statements hold for an object $U$ in $\T$: 
	\begin{enumerate}[label={(\arabic*)}]
			\item The morphism
			$(U,B) \xrightarrow{g_*} (U,C)$ is surjective if and only if $C$ is not a direct summand in $U$.
			\item The morphism \mbox{$(U,A) \xrightarrow{f_*} (U,B)$} is injective if and only if $~\Sigma^{-1}C~$ is not a direct summand in $U$. 
			\item The morphism
			$(B,U) \xrightarrow{f^*} (A,U)$ is surjective if and only if $A$ is not a direct summand in $U$.
			\item The morphism \mbox{$(C,U) \xrightarrow{g^*} (B,U)$} is injective if and only if $~\Sigma A~$ is not a direct summand in $U$. 
	\end{enumerate}
\end{lemma}

\begin{proof}
	Note that $C$ is a direct summand in $U$ if and only if there exists a split epimorphism $U \rightarrow C$. By the definition of an AR-triangle, this is equivalent to $g_*$ not being surjective, which proves \textit{(1)}. 
	
	Our triangle yields the long-exact sequence
	\[
	\cdots \rightarrow (U,\Sigma^{-1}B) \xrightarrow{(\Sigma^{-1}g)_*} (U,\Sigma^{-1}C) \rightarrow (U,A) \xrightarrow{f_*} (U,B) \rightarrow \cdots.
	\]
	The morphism $f_*$ is hence injective if and only if $(\Sigma^{-1}g)_*$ is surjective. By applying part \textit{(1)} to the object $\Sigma U$, we see that $(\Sigma^{-1}g)_*$ is surjective if and only if $C$ is not a direct summand in $\Sigma U$, which is equivalent to $\Sigma^{-1}C$ not being a direct summand in $U$. This shows \textit{(2)}.
	
	The statements \textit{(3)} and \textit{(4)} are verified dually, using that AR-triangles equivalently can be defined in terms of a factorization property for the leftmost morphism, see for instance \cite{happel3}.
\end{proof}

\begin{lemma} \label{Lemma AR part 2}
Let $A \xrightarrow{f} B \xrightarrow{g} C \rightarrow \Sigma A$ be an AR-triangle in $\T$. The following statements hold for an indecomposable object $U$ in $\T$: 
\begin{enumerate}[label={(\arabic*)}]
    \item We have $[U,A]-[U,B]+[U,C] \neq 0$ if and only if $U \simeq C$ or $U \simeq \Sigma^{-1}C$.
    \item We have $[A,U]-[B,U]+[C,U] \neq 0$ if and only if $U \simeq A$ or $U \simeq \Sigma A$.
\end{enumerate}
\end{lemma}

\begin{proof}
From the long exact Hom-sequence arising from our triangle, we get the exact sequence
\[
0 \rightarrow K \rightarrow (U,A) \xrightarrow{f_*} (U,B) \xrightarrow{g_*} (U,C) \rightarrow L \rightarrow 0,
\]
where $K= \Ker{f_*}$ and $L=\Coker{g_*}$. Splitting into short exact sequences and using our finiteness assumption, we see that the alternating sum of the lengths of the objects in the sequence vanishes. This gives the equation
\[
[U,A]-[U,B]+[U,C] = \operatorname{length}_R(K) + \operatorname{length}_R(L).
\]
Consequently, we have \mbox{$[U,A]-[U,B]+[U,C] \neq 0$} if and only if the right-hand side of the equation is also non-zero. This means that either $K$ or $L$ (or both) must be non-zero. The object $K$ is non-zero if and only if $f_*$ is not injective. By part \cref{Lemma AR part 1} part \textit{(2)}, this is the case if and only if $\Sigma^{-1} C$ is a direct summand in $U$. Similarly, the object $L$ is non-zero if and only if $g_*$ is not surjective. Using part \textit{(1)} of \cref{Lemma AR part 1}, this is equivalent to $C$ being a direct summand in $U$. As $U$ is indecomposable, a direct summand in $U$ is necessarily isomorphic to $U$, which finishes our proof of part \textit{(1)}.

Our second statement is shown dually, using part \textit{(3)} and \textit{(4)} of \cref{Lemma AR part 1}.
\end{proof}

We are now ready to prove our two main results, which show that we can study representation-finiteness of our category $\T$ by considering the relations for the associated Grothendieck group.

\begin{theorem} \label{main result}
	Assume there is an object $X$ in $\T$ such that $\Hom{\T}{Y}{X} \neq 0$ or an object $X'$ in $\T$ such that $\Hom{\T}{X'}{Y} \neq 0~$ for every non-zero $Y$ in $\T$. If $~\EXT = \ART$ in $K_0(\T,0)$, then $\T$ has only finitely many isomorphism classes of indecomposable objects.
\end{theorem}

\begin{proof}
Let $X$ be an object with the property described above, and consider the triangle $\Sigma^{-1}X \rightarrow 0 \rightarrow X \xrightarrow{1_X} X$. As this is a distinguished triangle, we have $[\Sigma^{-1}X]+[X] \in \EXT$. By the assumption $\EXT = \ART$, there hence exist AR-triangles
\[
A_i \rightarrow B_i \rightarrow C_i \rightarrow \Sigma A_i 
\]
and integers $a_i$ for $i=1,2,\dots,r$ such that 
\[
[X] + [\Sigma^{-1}X]=\sum_{i=1}^r a_i([A_i]-[B_i]+[C_i])
\]
in $K_0(\T,0)$. Given an object $U$ in $\T$, \cref{likning bevares} now yields the equality
\[
[U,X]+[U,\Sigma^{-1}X]=\sum_{i=1}^r a_i([U,A_i]-[U,B_i]+[U,C_i])
\]
in $\Z$. If $U$ is non-zero, our assumption on $X$ implies that the left-hand side of this equation is non-zero. Hence, there must for every non-zero object $U$ be an integer \mbox{$i\in \{1,\dots,r\}$} such that \mbox{$[U,A_i]-[U,B_i]+[U,C_i] \neq 0$}. In particular, this is true for every $U\in\indT$. By \cref{Lemma AR part 2} part \textit{(1)}, this means that any indecomposable object in $\T$ is isomorphic to an object in the finite set $\{C_i,\Sigma^{-1}C_i\}_{i=1}^{r}$, which yields our desired conclusion. 

The proof in the dual case is similar, using \cref{Lemma AR part 2} part \textit{(2)}.
\end{proof}

In the theorem below, an object $X$ in $\T$ is called a \textit{generator of $~\T$} if
\[
\operatorname{Hom}^*_{\T}(X,Y) = \bigoplus_{n \in \Z} \Hom{\T}{X}{\Sigma^n Y} \neq 0
\]
for any non-zero object $Y$ in $\T$. Dually, an object $X$ is called a \textit{cogenerator of $~\T$} if $\operatorname{Hom}^*_{\T}(Y,X) \neq 0$ for any non-zero $Y$. 

\begin{theorem} \label{main result2}
	Assume that our category $~\T$ has a generator or a cogenerator. If $~\EXT = \ART$ in $K_0(\T,0)$, then $\T$ has only finitely many isomorphism classes of indecomposable objects up to translation.
\end{theorem}

\begin{proof}
	Let $X$ be a cogenerator and consider an indecomposable object $U$ in $\T$. Notice that as $X$ is a cogenerator, there exists an integer $n$ such that $\Hom{\T}{\Sigma^n U}{X} \neq 0$. As in the proof of \cref{main result}, our assumption $~\EXT = \ART$ implies existence of a finite family of AR-triangles
	\[
	A_i \rightarrow B_i \rightarrow C_i \rightarrow \Sigma A_i 
	\]
	which yields an equality 
	\[
	[\Sigma^n U,X]+[\Sigma^n U,\Sigma^{-1}X]=\sum_{i=1}^r a_i([\Sigma^n U,A_i]-[\Sigma^n U,B_i]+[\Sigma^n U,C_i])
	\]
	in $\Z$. The left-hand side of this equation is non-zero, so there is an integer $i\in \{1,\dots,r\}$ such that $[\Sigma^n U,A_i]-[\Sigma^n U,B_i]+[\Sigma^n U,C_i] \neq 0$. By applying \cref{Lemma AR part 2} part \textit{(1)}, this yields that either $\Sigma^n U \simeq C_i$ or \mbox{$\Sigma^{n+1} U \simeq C_i$}. Consequently, every indecomposable object in $\T$ can be obtained as a translation of an object in the finite set $\{C_i\}_{i=1}^r$, which yields our desired conclusion.
	
	The proof in the case where our category $\T$ has a generator is dual, using \cref{Lemma AR part 2} part \textit{(2)}.
\end{proof}

\section{Application to Frobenius categories}

We now move on to an application of \cref{main result}. Throughout the rest of the paper, let $\C$ be an essentially small $R$-linear Frobenius category. Recall that a Frobenius category is an exact category with enough projectives and injectives, and in which these two classes of objects coincide. The stable category of $\C$, i.e.\ the quotient category modulo projective objects, is denoted \mbox{by $\underline{\C}$}. We assume $\C$ to be a Krull--Schmidt category and that the stable category $\underline{\C}$ is Hom-finite. 

As $\C$ is a Frobenius category, the associated stable category is triangulated. Recall that the distinguished triangles in $\underline{\C}$ are isomorphic to triangles of the form $~X \rightarrow Y \rightarrow Z \rightarrow \Omega^{-1} X,~$ where $~0 \rightarrow X \rightarrow Y \rightarrow Z \rightarrow 0~$ is a short exact sequence in $\C$ and $\Omega^{-1} X$ denotes the first cosyzygy of $X$. Note that $\Omega^{-1}$ is a well-defined autoequivalence on the stable category. The morphism $Z \rightarrow \Omega^{-1}X$ in our distinguished triangle above is obtained from the diagram
\[
    \begin{tikzcd}[column sep=15, row sep=20]
   0 \arrow[r] & X \arrow[r] \arrow[d,"1_X"] & Y \arrow[r] \arrow[d] & Z \arrow[r] \arrow[d, dashed] & 0 \\
   0 \arrow[r] & X \arrow[r] & I(X) \arrow[r] & \Omega^{-1}X \arrow[r] & 0,
    \end{tikzcd}
\]
where $I(X)$ is injective and both rows are short exact sequences. For a more thorough introduction to exact categories and the stable category of a Frobenius category, see for instance \cite{happel3}.

Based on the correspondence between short exact sequences in a Frobenius category and distinguished triangles in its stable category, we get results also for Frobenius categories. In order to see this, we need to rephrase some of our terminology in the context of exact categories. Let us first recall that a short exact sequence $~0 \rightarrow A \rightarrow B \xrightarrow{g} C \rightarrow 0~$ in $\C$ is an \textit{Auslander--Reiten sequence} if the following conditions are satisfied:
\begin{enumerate}
    \item $A,C\in\indC$;
    \item the sequence does not split;
    \item given any morphism $t\colon W \rightarrow C$ which is not a split-epimorphism, there is a morphism $t^\prime\colon W \rightarrow B$ such that $g \circ t^\prime = t$.
\end{enumerate}

Just as in the triangulated case, we let $K_0(\C,0)$ denote the free abelian group generated by isomorphism classes of objects in $\C$ modulo the subgroup generated by the set $\{[A\oplus B] - [A] - [B] \mid A,B\in\C\}$. Again, we can define the subgroups $\EXC$ and $\ARC$ of $K_0(\C,0)$, but now in terms of short exact sequences instead of distinguished triangles. Namely, we let $\EXC$ be the subgroup generated by the subset
\[
  \left\lbrace [X]-[Y]+[Z] \;\middle|\;
  \begin{tabular}{@{}l@{}}
   there exists a short exact sequence\\
   $0 \rightarrow X \rightarrow Y \rightarrow Z \rightarrow 0~$ in $\C$
   \end{tabular}
  \right\rbrace
\]
and $\ARC$ the subgroup generated by
\[
  \left\lbrace [X]-[Y]+[Z] \;\middle|\;
  \begin{tabular}{@{}l@{}}
   there exists an AR-sequence\\
$0 \rightarrow X \rightarrow Y \rightarrow Z \rightarrow 0~$ in $\C$
   \end{tabular}
  \right\rbrace.
\]

The next lemma describes a well-known correspondence between AR-sequences in $\C$ and AR-triangles in $\underline{\C}$, see \cite[Lemma 3]{Roggenkamp}. 

\begin{lemma} \label{lemmaAR}
An exact sequence $~0 \rightarrow A \rightarrow B \rightarrow C \rightarrow 0~$ in $\C$ is an AR-sequence in $\C$ if and only if the corresponding distinguished triangle $~A \rightarrow B \rightarrow C \rightarrow \Omega^{-1} A~$ in $\underline{\C}$ is an AR-triangle in $\underline{\C}$.
\end{lemma}

We are now ready to show the following lemma regarding the subgroups $\EXC$ and $\ARC$ of $K_0(\C,0)$ and the analogous subgroups of $K_0(\underline{\C},0)$.

\begin{lemma} \label{equalsubgroups}
If \mbox{$~\EXC = \ARC$} in $K_0(\C,0)$, then \mbox{$\EXCu = \ARCu$} in $K_0(\underline{\C},0)$.
\end{lemma}

\begin{proof}
Assume $\EXC = \ARC$ in $K_0(\C,0)$ and consider a distinguished triangle in $\underline{\C}$. As we work with isomorphism classes of objects, we can assume that our triangle is of the form $~X \rightarrow Y \rightarrow Z \rightarrow \Omega^{-1} X,~$ where \mbox{$~0 \rightarrow X \rightarrow Y \rightarrow Z \rightarrow 0~$} is a short exact sequence in $\C$. Since \mbox{$\EXC = \ARC$}, there exist AR-sequences $~0 \rightarrow A_i \rightarrow B_i \rightarrow C_i \rightarrow 0~$ and integers $a_i$ for $i=1,2,\dots,r$ such that
\[
[X]-[Y]+[Z]=\sum_{i=1}^r a_i([A_i]-[B_i]+[C_i])
\]
in $K_0(\C,0)$, and hence also in $K_0(\underline{\C},0)$. By \cref{lemmaAR}, the right-hand side of this equation is contained in $\ARCu$. Thus, we have shown that \mbox{$\EXCu \subseteq \ARCu$}. The reverse inclusion is clear. 
\end{proof}

We hence have the following corollary to \cref{main result}. 

\begin{corollary} \label{cor}
Assume there is an object $X$ in $\C$ such that $\Hom{\underline{\C}}{Y}{X} \neq 0$ or an object $X'$ in $\C$ such that $\Hom{\underline{\C}}{X'}{Y} \neq 0$ for every non-zero $Y$ in $\underline{\C}$. If ~$\EXC = \ARC$ in $K_0(\C,0)$, then the following statements hold:
\begin{enumerate}
    \item The category $\C$ has only finitely many isomorphism classes of non-projective indecomposable objects.
    \item If $\C$ has only finitely many indecomposable projective objects up to isomorphism, then $\C$ has only finitely many isomorphism classes of indecomposable objects.
\end{enumerate}
\end{corollary}

\begin{proof}
As $\C$ is an essentially small $R$-linear Krull--Schmidt category, the same is true for the stable category $\underline{\C}$. As $\EXC = \ARC$ in $K_0(\C,0)$, \cref{equalsubgroups} yields that $\EXCu = \ARCu$ in $K_0(\underline{\C},0)$. The result now follows from \cref{main result}.
\end{proof}

Let us consider the example where $R$ is a complete Gorenstein local ring with an isolated singularity. Recall that the category of Cohen--Macaulay $R$-modules is Frobenius. As $R$ is an isolated singularity, the associated stable category is Hom-finite, and completeness of $R$ yields the Krull--Schmidt property. By \cite[Lemma 2.1]{hiramatsu}, our category has an object which satisfies the assumption in the corollary above. Since $R$ is local, there are only finitely many isomorphism classes of indecomposable projective objects. Consequently, part \textit{(2)} of \cref{cor} yields that if the AR-triangles generate the relations for the Grothendieck group of this category, then $R$ has only finitely many isomorphism classes of indecomposable Cohen--Macaulay modules. This recovers \cite[Theorem 1.2]{hiramatsu} of Hiramatsu.

Note that one could, if preferred, state \cref{main result} and \cref{cor} in terms of taking the tensor product with $\Q$, as in the result of Kobayashi \cite[Theorem 1.2]{kobayashi}. Hence, also Kobayashi's conclusions are recovered from our approach in the case of a complete Gorenstein ring. 

\begin{acknowledgements}
The author would like to thank her supervisor Petter Andreas Bergh for helpful discussions and comments. She would also thank an anonymous referee for careful reading and suggestions which led to significant improvement of the paper.
\end{acknowledgements}


\begin{bibdiv}
\begin{biblist}
	
\bib{auslander1984}{article}{
	author={Auslander, Maurice},
	title={Relations for Grothendieck groups of Artin algebras},
	journal={Proc. Amer. Math. Soc.},
	volume={91},
	date={1984},
	number={3},
	pages={336--340},
}

\bib{auslander}{article}{
   author={Auslander, Maurice},
   title={Almost split sequences and algebraic geometry},
   conference={
      title={Representations of algebras},
      address={Durham},
      date={1985},
   },
   book={
      series={London Math. Soc. Lecture Note Ser.},
      volume={116},
      publisher={Cambridge Univ. Press, Cambridge},
   },
   date={1986},
   pages={165--179},
}

\bib{AR}{article}{
   author={Auslander, Maurice},
   author={Reiten, Idun},
   title={Grothendieck groups of algebras and orders},
   journal={J. Pure Appl. Algebra},
   volume={39},
   date={1986},
   number={1-2},
   pages={1--51},
}

\bib{AR2}{article}{
   author={Auslander, Maurice},
   author={Reiten, Idun},
   title={Representation theory of Artin algebras. III. Almost split
   sequences},
   journal={Comm. Algebra},
   volume={3},
   date={1975},
   pages={239--294},
}

\bib{ARS}{book}{
   author={Auslander, Maurice},
   author={Reiten, Idun},
   author={Smal\o , Sverre O.},
   title={Representation theory of Artin algebras},
   series={Cambridge Studies in Advanced Mathematics},
   volume={36},
   publisher={Cambridge University Press, Cambridge},
   date={1995},
   pages={xiv+423},
}

\bib{beligiannis}{article}{
   author={Beligiannis, Apostolos},
   title={Auslander-Reiten triangles, Ziegler spectra and Gorenstein rings},
   journal={$K$-Theory},
   volume={32},
   date={2004},
   number={1},
   pages={1--82},
}

\bib{butler}{article}{
   author={Butler, Michael C. R.},
   title={Grothendieck groups and almost split sequences},
   conference={
      title={Integral representations and applications},
      address={Oberwolfach},
      date={1980},
   },
   book={
      series={Lecture Notes in Math.},
      volume={882},
      publisher={Springer, Berlin-New York},
   },
   date={1981},
   pages={357--368},
}

\bib{DPW}{article}{
   author={Diveris, Kosmas},
   author={Purin, Marju},
   author={Webb, Peter},
   title={Combinatorial restrictions on the tree class of the
   Auslander-Reiten quiver of a triangulated category},
   journal={Math. Z.},
   volume={282},
   date={2016},
   number={1-2},
   pages={405--410},
}

\bib{enomoto}{article}{
	author={Enomoto, Haruhisa},
	title={Relations for Grothendieck groups and representation-finiteness},
	journal={J. Algebra},
	volume={539},
	date={2019},
	pages={152--176},
}

\bib{happel}{article}{
   author={Happel, Dieter},
   title={Auslander-Reiten triangles in derived categories of
   finite-dimensional algebras},
   journal={Proc. Amer. Math. Soc.},
   volume={112},
   date={1991},
   number={3},
   pages={641--648},
}

\bib{happel2}{article}{
   author={Happel, Dieter},
   title={On the derived category of a finite-dimensional algebra},
   journal={Comment. Math. Helv.},
   volume={62},
   date={1987},
   number={3},
   pages={339--389},
}

\bib{happel3}{book}{
   author={Happel, Dieter},
   title={Triangulated categories in the representation theory of
   finite-dimensional algebras},
   series={London Mathematical Society Lecture Note Series},
   volume={119},
   publisher={Cambridge University Press, Cambridge},
   date={1988},
   pages={x+208},
}

\bib{hiramatsu}{article}{
   author={Hiramatsu, Naoya},
   title={Relations for Grothendieck groups of Gorenstein rings},
   journal={Proc. Amer. Math. Soc.},
   volume={145},
   date={2017},
   number={2},
   pages={559--562},
}

\bib{jorgensen}{article}{
   author={J\o rgensen, Peter},
   title={Calabi-Yau categories and Poincar\'{e} duality spaces},
   conference={
      title={Trends in representation theory of algebras and related topics},
   },
   book={
      series={EMS Ser. Congr. Rep.},
      publisher={Eur. Math. Soc., Z\"{u}rich},
   },
   date={2008},
   pages={399--431},
}

\bib{kobayashi}{article}{
   author={Kobayashi, Toshinori},
   title={Syzygies of Cohen-Macaulay modules and Grothendieck groups},
   journal={J. Algebra},
   volume={490},
   date={2017},
   pages={372--379},
}

\bib{PPPP}{article}{
	author={Padrol, Arnau},
	author={Palu, Yann},
	author={Pilaud, Vincent},
	author={Plamondon, Pierre-Guy},
	title={Associahedra for finite type cluster algebras and minimal relations between $g$-vectors},
	journal={arXiv:1906.06861},
	date={2019},
}

\bib{Roggenkamp}{article}{
   author={Roggenkamp, Klaus W.},
   title={Auslander-Reiten triangles in derived categories},
   journal={Forum Math.},
   volume={8},
   date={1996},
   number={5},
   pages={509--533},
}

\bib{RVdB}{article}{
   author={Reiten, Idun},
   author={Van den Bergh, Michel},
   title={Noetherian hereditary abelian categories satisfying Serre duality},
   journal={J. Amer. Math. Soc.},
   volume={15},
   date={2002},
   number={2},
   pages={295--366},
}

\bib{Webb}{article}{
   author={Webb, Peter},
   title={Bilinear forms on Grothendieck groups of triangulated categories},
   conference={
      title={Geometric and topological aspects of the representation theory
      of finite groups},
   },
   book={
      series={Springer Proc. Math. Stat.},
      volume={242},
      publisher={Springer, Cham},
   },
   date={2018},
   pages={465--480},
}

\bib{XZ}{article}{
   author={Xiao, Jie},
   author={Zhu, Bin},
   title={Relations for the Grothendieck groups of triangulated categories},
   journal={J. Algebra},
   volume={257},
   date={2002},
   number={1},
   pages={37--50},
}

\bib{yoshino}{book}{
   author={Yoshino, Yuji},
   title={Cohen-Macaulay modules over Cohen-Macaulay rings},
   series={London Mathematical Society Lecture Note Series},
   volume={146},
   publisher={Cambridge University Press, Cambridge},
   date={1990},
   pages={viii+177},
}



\end{biblist}
\end{bibdiv}
\end{document}